\documentclass{article}
\usepackage{amsmath}
\usepackage{amscd}
\usepackage{amsthm}
\usepackage{amssymb}
\usepackage{romansections}
\renewcommand{\Bbb}{\mathbb}  
\renewcommand{\bar}{\overline}

\newcommand{\Q}{{\Bbb{Q}}}  
\newcommand{\R}{{\Bbb{R}}}  
\newcommand{\C}{{\Bbb{C}}}  
\newcommand{\Z}{{\Bbb{Z}}}  
\newcommand{\Pe}{{\Bbb{P}}} 
\newcommand{\A}{{\Bbb{A}}}  
\newcommand{\E}{{\Bbb{E}}}  
\newcommand{\G}{{\Bbb{G}}}  
\newcommand{\gm}[1]{\G_{m,#1}}      
\newcommand{\U}{\Bbb{U}}    
\newcommand{\Ve}{\Bbb{V}}
\renewcommand{\k}{\mathbf{K}}  
\newcommand{\shF}{{\cal F}} 
\newcommand{\shG}{{\cal G}} 
\newcommand{\shV}{{\cal V}}

\newcommand{\HH}{\underline{H}}    
\newcommand{\tensor}{\otimes}   
\newcommand{\isom}{\cong}       
\newcommand{\Hom}{\operatorname{Hom}}
\newcommand{\Ext}{\operatorname{Ext}}
\newcommand{\prolim}{\varprojlim}
\newcommand{\rtors}{\operatorname{tors}} 

\newtheorem{lemma}{Lemma}[subsection]

\newtheorem{prop}[lemma]{Proposition}
\newtheorem{theorem}[lemma]{Theorem}
\newtheorem{defn}[lemma]{Definition}

\newtheorem{conj}[lemma]{Conjecture}

\newcommand{\examples}{\noindent{\bf Examples:\ }}

\newcommand{\spec}{\operatorname{Spec}}
\newcommand{\rel}{\text{ rel }}        
\newcommand{\sgn}{{\text{sgn}}}        
\newcommand{\hm}{H_{\cal M}}           
\newcommand{\hgen}{H_{\operatorname{gen}}}    
\newcommand{\habs}{H_{\abs}}
\newcommand{\res}{\operatorname{res}}  
\newcommand{\cosq}{\operatorname{cosk}} 
\newcommand{\cone}{\operatorname{Cone}} 
\newcommand{\hodge}{{\cal H}}

\newcommand{\V}{\Bbb{V}} 

\newcommand{\Gr}{\operatorname{Gr}} \newcommand{\Log}{\operatorname{{\cal L\it og}}}
 
\newcommand{\abs}{\operatorname{abs}} 
 
 \newcommand{\pol}{\operatorname{pol}}
 \newcommand{\realH}{\operatorname{real}_{\Hh}}

 \newcommand{\Sym}{\operatorname{Sym}}

\newcommand{\Hh}{{\cal H}}  
\newcommand{\Lh}{{\cal L}} 
\newcommand{\Mh}{{\cal M}}  
\newcommand{\Oh}{{\cal O}}

\newcommand{\silo}{\stackrel{\sim}{\longrightarrow}}
\newcommand{\pfeil}{\longrightarrow}

\newcommand{\verk}{{\scriptstyle \circ}} \newcommand{\MHM}{\operatorname{\rm
MHM}} 
  
\newcommand{\MHS}{\operatorname{MHS}}

 \newcommand{\ea}{\mathfrak{a}}
\newcommand{\loc}{\operatorname{loc}}
 
 \newcommand{\Var}{\operatorname{Var}}
\newcommand{\For}{\operatorname{For}} \newcommand{\pr}{\operatorname{pr}}

\newcommand{\Gen}{{\cal G}en}

\newcommand{\Li}{\operatorname{Li}}

\newcommand{\mup}{\mu_d^0}

\begin{document}
\vspace*{3cm}
\begin{center}
{\LARGE \bf The Classical Polylogarithm}\\[0.5cm]
{\large Abstract of a series of lectures\\[0.1cm]
given at the workshop on polylogs in Essen,\\[0.1cm]
May 1 -- 4, 1997\\[0.3cm]
Annette Huber and J\"org Wildeshaus}
\end{center}
\vspace{5cm}
Annette Huber\hfill huber at math.uni-muenster.de\\
J\"org Wildeshaus\hfill wildesh at math.uni-muenster.de\\
Math. Institut,
Einsteinstr. 62,
48 149 M\"unster,
Germany\\

\begin{center}{\it address update 2012:}\end{center}

\noindent Annette Huber \hfill annette.huber at math.uni-freiburg.de\\
Math. Institut, Universit\"at Freiburg, Eckerstr. 1, Germany\\

\noindent J\"org Wildeshaus\hfill wildesh at math.univ-paris13.fr\\
D\'epartement de Math\'ematiques, Institut Galil\'ee, Universit\'e Paris 13, \\
99, avenue Jean-Baptiste Cl\'ement, 93430 - Villetaneuse, France

\newpage
\tableofcontents
\newpage
\noindent The purpose of this series of lectures was to give an overview of the
central ideas in the proof of the following\\

\noindent{\bf Main Theorem.} {\it Let $d \ge 2$, $j \ge 2$, $\mup$ the set of primitive
$d$-th roots of unity in $F:= \Q (\mu_d)$. There is a (unique) map of sets
\[
\epsilon_{j}: \mup 
\longrightarrow K_{2j-1}(F)_\Q := K_{2j-1}(F) \otimes_{\Z} \Q 
\]
whose composition with the regulator to Deligne, or absolute Hodge cohomology
\[
r_{\Hh}: K_{2j-1}(F)_\Q \longrightarrow 
\left( \bigoplus_{\sigma: F \hookrightarrow \C} 
\C / (2 \pi i)^j \Q \right)^+
\]
sends $\omega \in \mup$ to the element 
\[
\left(- Li_j(\sigma \omega) \right)_{\sigma} = \left(- \sum_{k \ge 1} \frac{\sigma \omega^k}{k^j} \right)_{\sigma} \; .
\]
Here, $+$ denotes the invariant part under the joint operation of complex
conjugation on the set of embeddings of $F$ into $\C$, and on 
$\C/(2\pi i)^j\Q$. }\\

The result is due to Beilinson (\cite{B1}, 7.1.5). The original proof, somewhat
sketchy, is beautifully reviewed in \cite{Neu}. It relies heavily on a result 
on the explicit shape of a construction called the ``Loday symbol'' in Deligne
cohomology. This so-called ``Crucial Lemma'' (\cite{Neu}, II.2.4) was subsequently proved in \cite{E}, 3.9.\\

The proof given in the course of this series is different from the original
one, and makes use of the classical, or cyclotomic {\it polylogarithm}. One
of the great advantages of this approach is that the complicated calculations
in Deligne cohomology are no longer necessary. In fact, the polylog enjoys a
characteristic property called {\it rigidity}. One of the aims of the lectures
was to emphasize the r\^ole of rigidity played in the explicit representation
of the objects.\\

Let us remark that the Main Theorem admits an $l$-adic counterpart: \cite{HW},
Corollary 9.7. The statement was conjectured by Bloch and Kato (\cite{BK},
Conjecture 6.2), and here the only complete proof is via polylogarithms.
That the talks concentrate on the Hodge theoretic aspects of the theory has to
do with the speakers' desire to make the objects as ``visible'' as possible --
to their taste, this requirement is satisfied to a larger degree of satisfaction
by the objects of Hodge theory rather than those of the \'etale world.
The main strategy of proof, and the abstract concepts however admit immediate
translations to the $l$-adic setting.\\

The main ideas of the proof of the Main Theorem, and its $l$-adic counterpart,
appear already in the preprint \cite{B2}. Since then, quite a lot of
polylogarithmic literature has been published. The speakers like to think of
their talks, and in fact, of this abstract, as a guide through the literature.
We hope that it will be useful particularly to those who are new to the field.\\

We would like to thank the organizers of the workshop, G.\ Frey, H.\ Gangl,
and H.-G.\ R\"uck, for the invitation to Essen. To us, it meant a great
opportunity to give a self-contained exposition of some central aspects
of the theory.

\section{Motivation}
In this talk, it was tried to indicate the main strategy of proof:
Consider, for a number field $F$, the regulator
\begin{equation}\tag{*}
r_{\Hh}: K_{2j-1}(F)_\Q \longrightarrow 
\left( \bigoplus_{\sigma: F \hookrightarrow \C} 
\C / (2 \pi i)^j \Q \right)^+\; . 
\end{equation}
The zeroeth step is to introduce the concept of ``Yoneda extensions in
categories of mixed sheaves'' in order to reinterpret both sides of $(*)$.
In lecture II, we shall define, for a smooth and separated $\R$-scheme $X$,
a $\Q$-linear tensor category
\[
\Var(X/\R)
\]
of {\it variations of mixed $\Q$-Hodge structure over $\R$ on $X$}.\\

The reinterpretation of the right hand side of $(*)$ acquires the following
shape:\\

\noindent{\bf Proposition.} {\it Let $j \ge 1$. There is a canonical isomorphism
\[ 
\Ext^1_{\Var \left( F \otimes_{\Q} \R / \R \right)} (\Q(0),\Q(j)) \silo
\left( \bigoplus_{\sigma:F\hookrightarrow\C}\C/(2\pi i)^j\Q \right)^+ \]
where $+$ denotes the invariant part under the joint operation of complex
conjugation on $\{ F \hookrightarrow \C \}$, and on $\C/(2\pi i)^j\Q$. }\\

According to the motivic folklore, there should be a $\Q$-linear tensor
category of {\it smooth mixed motivic sheaves} $\Mh\Mh^s (X)$ on any smooth and
separated scheme $X$ over $\Q$, together with an exact tensor functor, called the
{\it Hodge realization}
\[
\realH: \Mh\Mh^s (X) \longrightarrow 
\Var \left( X \times_{\Q} \R / \R \right) \; .
\]
There should be an isomorphism
\[
\Ext^1_{\Mh\Mh(\spec F)} (\Q(0),\Q(j)) \silo
K_{2j-1}(F)_\Q
\]
for $j \ge 1$ identifying the regulator $r_{\Hh}$ with the morphism induced by
$\realH$. This would give the sheaf theoretical reinterpretation of the left
hand side of $(*)$, and in fact, of $r_{\Hh}$.\\

In first approximation, the proof of the Main Theorem, and indeed, also of its
$l$-adic counterpart, proceeds in two steps, corresponding to lectures II--V,
and VI--VII respectively:\\

\noindent 1. Construct
\[
r_{\Hh} \verk \epsilon_j: \mup \longrightarrow
\Ext^1_{\Var \left( F \otimes_{\Q} \R / \R \right)} (\Q(0),\Q(j))
\]
first.
\begin{enumerate}
\item[a)]
The construction is a priori sheaf theoretical, and uses concepts like
Leray spectral sequences. The objects will be characterized by certain
universal properties, one consequence of which will be the earlier mentioned
rigidity.
\item[b)]
Via rigidity, it is possible to describe explicitly the objects defined by
abstract nonsense. In particular, we get the formula of the Main Theorem for
$r_{\Hh} \verk \epsilon_j (\omega)$, $\omega \in \mup$.
\item[c)]
Again via rigidity, it is possible to show that the abstract construction of
a) is ``geometrically motivated'': the one-extensions $r_{\Hh} \verk \epsilon_j (\omega)$ occur as cohomology objects, with Tate coefficients, of certain
$F$-schemes.
\end{enumerate}
2. Because of the present non-availability of a sheaf theoretical machinery
on the level of motives, step 1.a) cannot simply be imitated. However, it
turns out that 1.c) admits a translation to $K$-theory, yielding the map
\[
\epsilon_{j}: \mup 
\longrightarrow K_{2j-1}(F)_\Q \; .
\]
Its compatibility with the map $r_{\Hh} \verk \epsilon_j$ under the regulator
is then a consequence of the definition.\\

Let us give a more detailed account of step 1. Again, we owe to Beilinson the
insight that instead of treating the $r_{\Hh} \verk \epsilon_j (\omega)$,
$\omega \in \mup$, $d,j \ge 2$ separately, one should construct one object
containing all the information.\\

In lecture III, we are going to define, by some universal property,
the {\it logarithmic (pro-)variation} $\Log$ on $\G_{m,\R}$. We have
\[
\Gr_*^W\Log=\prod_{j\geq 0}\Q(j) \; ,
\]
i.e., $\Log$ is a successive extension of $\Q(0)$ by $\Q(1)$ by $\Q(2)$...
The {\it polylogarithmic extension} $\pol$ is a one-extension, in the category
of variations on the $\R$-scheme $\U_\R:=\Pe^1_\R \setminus\{0,1,\infty\}$, of $\Q(0)$ by the
restriction of $\Log$:
\[
\pol\in\Ext^1_{\Var(\U_\R /\R)} \left( \Q(0),\Log_{\U_\R} \right) \; .
\]
Again, the definition is via a universal property. From it, we deduce rigidity,
and are consequently able, in talk IV, to give an explicit description of $\pol$, essentially in terms of (the inverse of) its period matrix. This description will then justify the name ``polylogarithm'' as the entries of
this matrix are essentially given by the higher logarithms.\\

In lecture V, we establish another characteristic feature of our objects: the
so-called {\it splitting principle}. Any $\omega \in \mup$ induces an embedding
\[
i_{\omega}: \spec  \Q(\mu_d)  \hookrightarrow \U \; .
\]
\noindent{\bf Theorem.}{\it 
$i_{\omega}^* \Log$ is canonically split:}
\[
i_{\omega}^* \Log = \prod_{j\geq 0}\Q(j) \; .
\]

Therefore, we may think of $i_{\omega}^* \pol$ as an element of
\[
\prod_{j\geq 1} \Ext^1_{\Var( \Q(\mu_d) \otimes_{\Q} \R /\R )}(\Q(0),\Q(j))
= \prod_{j\geq 1} \left( \bigoplus_{\sigma} 
\C / (2 \pi i)^j \Q \right)^+ \; .
\]
From the explicit description of $\pol$, it is straightforward to see that the
$j$-th component of $i_{\omega}^* \pol$ equals, up to scaling, the element
\[
\left(- Li_j(\sigma \omega) \right)_{\sigma} \; .
\]
In this very precise sense, all the $r_{\Hh} \verk \epsilon_j (\omega)$,
$\omega \in \mup$ are interpolated by $\pol$. The additional data is the
action of the fundamental group of $\U (\C)$ given by the local system
underlying the variation $\pol$. In fact, rigidity is formulated in terms of
this local system. \\

In talk VI, we give a sketch of some of the technical ingredients for a 
suitable formalism of ``relative $K$-theory''. Lecture VII establishes the
geometric realization of $\Log$, and of $\pol$ in absolute and motivic
cohomology.

\section{Review of Hodge-theory}
We assemble some facts from Hodge theory that are needed in the construction
of the polylogarithm. Basically everything (including references) is
contained in \cite{BZ}.

\subsection{Mixed Hodge structures}
We are mostly interested in Hodge structures of Tate type, i.e., ones where
only  Hodge numbers $(n,n)$ for $n\in\Z$ occur. 
\begin{lemma} \label{II.1.1} There is a natural isomorphism for $n> 0$
\[\Ext^1_{\MHS}(\Q(0),\Q(n))\isom \C/(2\pi i)^n\Q \]
which assigns to $s\in\C$ the extension class of $E_s$ where
$E_{s\C}=\C^2$ and $E_{s\Q}\subset \C^2$ is given by 
\[\left(\begin{array}{cc}1&0\\ -\frac{s}{(2\pi i)^n}&1\end{array}\right)
\left(\begin{array}{c}\Q\\(2\pi i)^n\Q\end{array}\right).\]
with weight and Hodge filtration in $\C^2$ given by
\[
W_i=\begin{cases}\left(\begin{array}{c}0\\ 0\end{array}\right)&i<-2n,\\[3ex]
              \left(\begin{array}{c}0\\ *\end{array}\right)&-2n\leq i<0,\\[3ex]
               \left(\begin{array}{c}*\\ *\end{array}\right)&0\leq i,
     \end{cases}\hspace{5ex}
F^p=\begin{cases}\left(\begin{array}{c}*\\ *\end{array}\right)&p\leq -n,\\[3ex]
              \left(\begin{array}{c}*\\ 0\end{array}\right)&-n<p\leq 0,\\[3ex]
               \left(\begin{array}{c}0\\ 0\end{array}\right)&0< p.
     \end{cases}
\]
\end{lemma}
\begin{proof}E.g.\cite{J} Lemma 9.2 and Remark 9.3.a).\end{proof}
\examples\begin{enumerate}
\item
In particular, we have 
\[ \Ext^1_{\MHS}(\Q(0),\Q(1))\isom \C/2\pi i\Q\xrightarrow{exp} \C^*\tensor\Q.
\]
We reinterpret this relation by saying that we assign to each $z\in\G_m(\C)$
a mixed Hodge structure called $\Log^{(1)}(z)$, namely $E_{\log(z)}$ in the
notation of the lemma. 
\item For $z\in\C^*$ we have a long exact sequence in $\MHS$:
\[0\to\HH^0(\G_m(\C))\xrightarrow{\Delta}{}\HH^1(\G_m(\C)\rel \{1\}\amalg\{z\})
\to \HH^1(\G_m(\C))\to 0. \]
($\HH$ denotes the corresponding singular cohomology as mixed
Hodge structure.) Put $\shG^{(1)}(z)=\HH^1(\G_m(\C)\rel \{1\}\amalg\{z\})$. This
is again an element in $\Ext^1_{\MHS}(\Q(0),\Q(1))$. We will see in  lecture III that
$\shG^{(1)}(z)=\Log^{(1)}(z)$.
\end{enumerate}

\subsection{Variations}
Let $X$ be a smooth complex algebraic variety.
\begin{defn} A variation $\V$ of mixed Hodge structure on $X(\C)$ consists of 
\begin{itemize}
\item a locally constant sheaf $\Ve_\Q$ of $\Q$-vector spaces on $X(\C)$,
\item an increasing filtration $W_*$ of $\Ve_\Q$ by locally constant
sheaves,
\item a decreasing filtration $\shF^*$ of $\shV=\Ve_\Q\tensor_\Q\Oh_X$ by
holomorphic subvector bundles,
\end{itemize}
such that for each $x\in X$, the data induce a mixed Hodge structure on
$\Ve_x$ and such that Griffith transversality holds. We denote the
category $\Var(X(\C))$.
A variation is called {\em unipotent} if all $\Gr^W_n\Ve$ are constant on $X$, e.g.,
all variation of Tate Hodge structure are.
\end{defn}
\examples $\Log^{(1)}$ and $\shG^{(1)}$ are unipotent variations on $\G_m(\C)$.

\begin{lemma} \label{II.2.2}
\[\Ext^1_{\Var(X(\C))}(\Q(0),\Q(1))\isom\Oh_{\operatorname{hol}}^*(X(\C))\tensor\Q.\]
\end{lemma}
\begin{proof}E.g.\cite{W} IV Theorem 3.7. a) \end{proof}
There is a notion of {\em admissible} variations of Hodge structure. Rather
than giving the definition, we give their main properties which will 
suffice for all that follows. We denote the corresponding category
$\Var(X)$ to stress its algebraic nature.
\begin{enumerate}
\item
If $X$ is compact, then all variations are admissible.
\item
Everything coming from geometry is admissible, e.g., $\shG^{(1)}$ is.
\item If $U$ is an algebraic variety and $X$ a smooth proper compactification, then (\cite{W} IV Theorem 3.7 b)):
\begin{align*}
\Ext^1_{\Var(U)}(\Q(0),\Q(1))&\isom
\left\{g\in\Oh^*_{\operatorname{hol}}(U)\tensor\Q \mid \text{$g$ meromorphic on $X$}\right\} \\
&=\Oh^*_{\operatorname{alg}}(U)\tensor\Q.
\end{align*}
\item
If $\Ve$ is admissible on $X$, then all $H^i(X,\Ve)$ carry a canonical
mixed Hodge structure. This is a deep result, due to Steenbrink-Zucker in
the case of curves, and M. Saito in general.
\item
If $Y\subset X$ is an immersion of smooth varieties of pure codimension $d$,
$U=X\setminus Y$ and $\Ve\in\Var(X)$, then there is a natural long
exact sequence in $\MHS$:
\end{enumerate}
\[
\cdots\to \HH^{i-2d}(Y,\Ve_Y(-d))\to\HH^i(X,\Ve)\to
\HH^i(U,\Ve_U)\to\HH^{i+1-2d}(Y,\Ve_Y(-d))\to\cdots.
\]
\subsection{Everything over $\R$}
The reference for the following is \cite{HW} Appendix A.2. Let $X$ be a
smooth algebraic variety over $\R$. Then there is a continuous map
$\iota:X(\C)\to X(\C)$ given by complex conjugation on points. It induces
a functor
\begin{align*}
 \iota^*:\Var(X(\C))&\to\Var(X(\C))\\
 (\Ve,W_*,\shF^*)&\mapsto (\iota^*\Ve,\iota^*W_*,\iota^*\overline{\shF}^*).
\end{align*}
\begin{defn}
An admissible variation of mixed Hodge structure defined over $\R$ is
a pair $(\Ve,F_\infty)$ where $\Ve\in\Var(X_\C)$ and 
$F_\infty:\Ve\to\iota^*\Ve$ is an involution, i.e., $F_\infty^{-1}=\iota^*F_\infty$.
\end{defn}
\begin{lemma}
Let $U$ be a smooth variety over $\R$ and $X$ a smooth compactification. Then
\[\Ext^1_{\Var(\U/\R)}(\Q(0),\Q(1))=\Oh_{\operatorname{alg}}^*(U)\tensor\Q.\]
\end{lemma}
In particular, we need the case
$X=\spec K\tensor_\Q\R$ where $K$ is some finite extension of $\Q$. Note that
then $X(\C)=\coprod_{\sigma:k\to\C}\operatorname{point}$. Hence
\[ \Ext^1_{\Var(X/\R)}(\Q(0),\Q(n))=
\left(\bigoplus_{\sigma:K\to\C}\C/(2\pi i)^n\Q\right)^+ \]
where $+$ denotes the invariant part under the joint operation of complex
conjugation on $X(\C)$ and $\C/(2\pi i)^n\Q$.

\section{The Logarithm and the polylogarithm}
The aim of this lecture was to construct two things:
\begin{description}
\item[a)] a (pro)-object $\Log$ in $\Var(\G_{m,\R}/\R)$ such that (at least)
\[\Gr_*^W\Log=\prod_{n\geq 0}\Q(n)\text{\ and\ }
\Log_1=\prod_{n\geq 0}\Q(n).\]
\item[b)]
on $\U_\R=\Pe^1_\R \setminus\{0,1,\infty\}$ an element
\[\pol\in\Ext^1_{\Var(\U_\R/\R)}(\Q(0),\Log_\U).\]
\end{description}
There are three possibilities to do this: explicitly (talk IV), 
geometrically (end of III and VII) and by a universal property. It is
this last method that we describe first.
\subsection{The logarithm} \label{III.1}
\begin{theorem}[Chen, \cite{BZ} 6.23] Let $X$ be a smooth algebraic variety over
$\C$. Let $x\in X(\C)$ and $\pi=\pi_1(X(\C),x)$. We denote
$U=\Q[\pi]$ and its augmentation ideal $\ea$. Then the completion
$\hat{U}=\prolim U/\ea^n$ carries a (unique) mixed Hodge structure
such that
\begin{itemize}
\item
$\hat{U}\tensor\hat{U}\xrightarrow{\operatorname{mult}}{}\hat{U}$ and
$\operatorname{unity}:\Q(0)\to\hat{U}$ are morphisms of mixed Hodge
structures; 
\item
there is an isomorphism of mixed Hodge structures
\[ \ea/\ea^2\leftarrow\pi_1(X(\C),x)^{\operatorname{ab}}\tensor\Q\isom
H_1(X(\C),\Q).\]
\end{itemize}
\end{theorem}
Our example is $X=\G_m$, $x=1$ and hence $\pi=\Z\gamma$ with a positively
oriented loop around $0$. Then $\ea$ is generated by $\gamma-1$. We get
???? 
$\hat{U}=\Q[e]$ where $e=\log\gamma$. Note that the latter element is
defined in the completion. The mixed Hodge structure on 
$\ea/\ea^2\isom H_1(\C^*,\Q)\isom\Q(1)$ is in fact pure. All in all
\[ \hat{U}=\Sym^*\ea/\ea^2=\prod_{n\geq 0}\Q(n). \]
\begin{theorem}[Hain-Zucker, \cite{BZ} 7.19]Let $X$ be a smooth connected
algebraic variety over $\C$. Then there is an equivalence of categories
\begin{gather*}
\{ \text{admissible unipotent variations on $X$} \}\\
\updownarrow\\
\left\{\begin{array}{l}
V\in\MHS\text{ together with an operation }\hat{U}\tensor V\to V\\
\text{which is a morphism of $\MHS$.}
\end{array}\right\}
\end{gather*}
where we assign to a variation its monodromy representation on the
stalk at $x$.
\end{theorem}
\begin{defn}[\cite{W} p.43]
$\Gen_x$, the generic variation based at $x$ is the variation corresponding
to the representation $\hat{U}\tensor\hat{U}\to\hat{U}$ given by multiplication.
\end{defn}
The generic variation has a universal property. Let 
\[ \Gamma:\MHS\to \text{$\Q$-vector spaces}\]
 be the global section functor, i.e., 
$\Gamma(H)=\Hom_{\MHS}(\Q(0),H)=W_0H_\Q\cap F^0H_\C$.
\begin{prop} \label{III.1.4}
The pair $(\Gen_x,1\in \Gamma(\hat{U}))$ (pro)-represents the functor
\[ \Gamma(x^*?):\{\text{unipotent objects in $\Var(X)$}\}\longrightarrow
\text{$\Q$-vector spaces.} \]
\end{prop}
\begin{proof} This follows immediately from the Theorem of Hain and Zucker.
\end{proof}

We can now identify the object we were after:
\begin{defn}[\cite{W} p. 94]
Let the {\em logarithmic sheaf} on $\G_m$ be $\Log=\Gen_1$. It has the above universal property
with respect to the stalk at $1$, which is $\prod_{i\geq 0}\Q(i)$.
Let $\Log^{(n)}=\Log/W_{-2n-2}\Log$ be the quotients of finite length.
\end{defn}
In fact $\Log$ is easily seen to be defined over $\R$.

\subsection{The polylogarithm extension} \label{III.2}
From the Leray spectral sequence for the composition of functors
\[ \Hom_{\Var(\G_m/\R)}(\Q(0),?)=\Hom_{\MHS/\R}(\Q(0),\HH^0(\G_m(\C),?))\]
 we
get the short exact sequence
\begin{multline*}
0\to\Ext^1_{\MHS/\R}(\Q(0),\HH^0(\U(\C),\Log_\U)) \to
\Ext^1_{\Var(\U/\R)}(\Q(0),\Log_\U)\\
\to \Hom_{\MHS/\R}(\Q(0),\HH^1(\U,\Log_\U)) \to 0.
\end{multline*}
\begin{lemma} \label{III.2.1}
\begin{align*}
\HH^0(\U,\Log_\U)&=\HH^0(\G_m,\Log)=0,\\
\HH^1(\U,\Log_\U)&=\Q(-1)\oplus\Log_1(-1)=\Q(-1)\oplus\prod_{k\geq 0}\Q(k-1).
\end{align*}
The map $\HH^1(\U,\Log_\U)\to \Log_1(-1)$ is residue at the point $1$.
\end{lemma}
Hence 
\[\Ext^1_{\Var(\U/\R)}(\Q(0),\Log_\U)=
       \Hom_{\MHS/\R}(\Q(0),\Q(-1)\oplus\Log_1(-1))=\Q\; .\]
\begin{defn}
We define the {\em polylogarithmic extension} 
\[\pol\in\Ext^1_{\Var(\U/\R)}(\Q(0),\Log_\U)\]
 as the preimage of $1$ under the above identification. In terms
of the group $\Hom_{\MHS/\R}(\Q(0),\Q(-1)\oplus\Log_1(-1))$ it is given
by $1\mapsto e\tensor (2\pi i)^{-1}$.
\end{defn}

\subsection{Geometric origin of $\Log^{(1)}$}
Recall the variation $\shG^{(1)}$ with fibre 
$\HH^1(\G_m(\C)\rel \{1\}\amalg \{z\},\Q(1))$ at $z\in\C^*$. It is unipotent and
admissible. At $z=1$, the short exact sequence of mixed Hodge structures
\[ 0\to \Q(1)^2/\Delta(\Q(1))\to \shG^{(1)}_1\to \Q(0) \to 0\]
has a splitting by $\HH^1(\G_m(\C)\rel\{1\},\Q(1))\isom\Q(0)$. This defines
a global section of $\shG^{(1)}_1$. By the universal property of $\Log$, there
is a canonical morphism 
\[\phi:\Log\to\shG^{(1)}\]
compatible with the projection to $\Q(0)$.
\begin{prop}[\cite{HW} Theorem 4.11]\ \\
$\phi$ induces an isomorphism on $\Log^{(1)}$.
\end{prop}
\begin{proof} The morphism factors for weight reasons. It is enough to check
that it induces an isomorphism on the underlying local systems.
Note that both objects are $2$-dimensional. The image
is at least one-dimensional. If is was indeed just one-dimensional, then
$\shG^{(1)}$ would be split and hence constant. So all we have to see is
whether $\shG^{(1)}$ has non-trivial monodromy. This is not hard to
do explicitly.
\end{proof}


\section{Explicit description of the polylog}
In this lecture, we use the abstract definition of the polylog to deduce
its main characteristic property, the so-called {\it rigidity principle}.
We then determine the explicit shape of $\pol$.

\subsection{Rigidity} 
Denote by $\loc(M)$ the category of local systems in finite dimensional
$\Q$-vector spaces on a topological space $M$, and by
\[
\For: \Var (X / \R) \longrightarrow \loc(X(\C))
\]
the forgetful functor.
\begin{theorem}[\cite{B2}, 2.1, \cite{W}, III, Theorem 2.1] \label{IV.1.1}
$\pol$ is uniquely determined by 
\[
\For(\pol) \in \Ext^1_{\loc(\U(\C))} \left( \Q, \For(\Log_\U) \right) \; .
\]
\end{theorem}
\begin{proof} There is a commutative diagram of boundary morphisms in Leray
spectral sequences
\[
\begin{CD}
\Ext^1_{\Var(\U /\R)} \left( \Q(0),\Log_{\U} \right)@>>>
\Hom_{\MHS/\R} \left( \Q(0),\HH^1(\U,\Log_\U) \right)\\
@V\For VV @VV\For V\\
\Ext^1_{\loc(\U(\C))} \left( \Q, \For(\Log_\U) \right)@>>> 
\Hom_\Q \left( \Q, H^1(\U(\C), \For(\Log_\U)) \right)
\end{CD}
\]
By \ref{III.2}, the upper horizontal map is an isomorphism. Since $\For$ is
injective on the level of homomorphisms, we see that the left vertical map
is injective, too.
\end{proof}

Recall from Lemma \ref{III.2.1} that
\[
\HH^1(\U,\Log_\U)=\Q(-1)\oplus\Log_1(-1)=\Q(-1)\oplus\prod_{k\geq 0}\Q(k-1) \;,
\]
and that $\HH^0(\U,\Log_\U) = 0$. It follows as in \ref{III.2} that the lower
horizontal map of the diagram in the proof of the theorem is an isomorphism,
and that it maps the class of $\For(\pol)$ to the morphism
\begin{align*}
\Q &\longrightarrow H^1(\U(\C), \For(\Log_\U)) = 
\For \left( \Q(-1)\oplus\prod_{k\geq 0}\Q(k-1) \right) \; ,\\
1 &\longmapsto \frac{1}{2 \pi i} \cdot e \; .
\end{align*}

The $l$-adic counterpart of this result replaces $\Var$ by the category of
lisse $l$-adic sheaves on a smooth scheme over a number field, and $\loc$
by the category of lisse $l$-adic sheaves on the base change by the algebraic
closure of the field.

\subsection{The local system underlying $\pol$}
Denote by $\widetilde{\pi}$ the fundamental group of $\U(\C)$ at the base point
$\frac{1}{2}$. It is free in the generators $\alpha_0$ and $\alpha_1$, where $\alpha_i$ denotes the class of the positively oriented circle around $i$.\\

Define the following representation of $\widetilde{\pi}$:
\begin{align*}
\E := &<1>_\Q \oplus <e^k, \; k \ge 0>_\Q \; ,\\
\alpha_0: 1 &\longmapsto 1 \; ,\\
        e^k &\longmapsto e^k \cdot \exp(e) \; .\\
\alpha_1: 1 &\longmapsto 1 + e \; ,\\
        e^k &\longmapsto e^k \; .
\end{align*}
We get an extension of $\widetilde{\pi}$-modules, i.e., of local systems on
$\U(\C)$
\[
0 \pfeil \For(\Log_\U) \pfeil \E \pfeil \Q \pfeil 0
\]
(recall from \ref{III.1} that $e = \log \alpha_0$).\\

From the remark following Theorem \ref{IV.1.1}, one concludes:
\begin{prop} [\cite{B2}, 2.1, \cite{W}, IV, Theorem 2.2] \label{IV.2.1}
The class of the above extension equals $\pol$.
\end{prop}

\subsection{Extensions of variations of Tate-Hodge structure}
We need to develop a language in which we can describe variations explicitly.
The following will be crucial:
\begin{theorem} Let $\left( \Ve_\Q, W_*, \shF^* \right) \in \Var (X(\C))$
be a variation of THS (Tate-Hodge structure). Then the underlying bifiltered
vector bundle
\[
\left( \shV, W_*, \shF^* \right)
\]
is canonically split. $\shV$ and all $\shF^p W_{2p} \shV$ are generated by global sections.
\end{theorem}
\begin{proof}
Since $\Gr^p_{\shF^*} \Gr_n^{W_*} \shV = 0$ for $n \ne 2p$, we have
\[
\shV = \bigoplus_p \shF^p W_{2p} \shV \; .
\]
For any $p$, we have canonically 
\[
\shF^p W_{2p} \shV \silo \Gr^{2p}_{W_*} \shV \; ,
\]
which is constant.
\end{proof}

This gives our recipe for describing variations of THS: Let 
$\left( \Ve_\Q, W_*, \shF^* \right)$ be one such.
\begin{enumerate}
\item Choose a basis of global sections of $\shV$ respecting the decomposition
\[
\shV = \bigoplus_p \shF^p W_{2p} \shV \; .
\]
\item Express a basis of $\Q$-rational flat sections respecting the weight
filtration $W_* \Ve_\Q$ in the basis of 1.
\end{enumerate}
The result will be a lower triagonal matrix. Its entries will in general 
consist of multivalued functions
since sections of $\Ve_\Q$ will usually only exist on the universal cover of
$X(\C)$.\\

Actually, we already applied this recipe: If $X$ is a point, Lemma \ref{II.1.1}
determines one-extensions of $\Q(0)$ by $\Q(n)$. Lemma \ref{II.2.2} describes
one-extensions of $\Q(0)$ by $\Q(1)$ for arbitrary $X$.

\subsection{Explicit shape of $\pol$}
In order to write down a matrix describing $\pol$ in the sense of the previous
section, we need to define some multivalued functions:
\begin{defn} 
\begin{align*}
\Li_1 (t) &:= - \log (1-t) \; ,\\
\Li_{k+1} (t) &:= \int_0^t \frac{\Li_k(s)}{s} ds \; , \quad k \ge 1 \; ,\\
\Lambda_k &:= \frac{1}{(-2 \pi i)^k} 
\sum^k_{n=1} \frac{(- \log)^{k-n}}{(k-n)!} \Li_n \; .
\end{align*}
\end{defn}
Using \ref{IV.2.1}, one proves:
\begin{lemma} [\cite{W}, IV, Lemma 3.3]
\[
f := 1 + \sum^{\infty}_{k=1} \Lambda_k \cdot e^k
\]
is a global section of
\[
\For(\pol) \otimes_\Q \Oh(\U(\C)) \; .
\]
\end{lemma}
So in the basis of global sections $(f,e_0,e_1,\dots)$, where
\[
e_k: t \longmapsto e^k \cdot \exp \left( \frac{\log(t)}{2 \pi i} \cdot e \right)
= e^k + \frac{\log(t)}{2 \pi i} \cdot e^{k+1} + \dots \; ,
\]
the rational structure is described by the following matrix $P$:
\[ 
\left( \begin{array}{cccccc}
1 & 0 & 0 & 0 & \!\cdots \\
0 & 1 & 0 & 0 & \!\cdots \\
\frac{1}{2\pi i} \Li_1 & - \frac{1}{2 \pi i} \log & 1 & 0 & \!\cdots \\
-\frac{1}{(2\pi i)^2} \Li_2 &
\frac{1}{2!} \left( -\frac{1}{2\pi i} \log \right)^2 &
-\frac{1}{2\pi i} \log & 1 & \!\cdots \\
\hspace*{0.2cm} \frac{1}{(2\pi i)^3} \Li_3 \hspace*{0.2cm} & 
\hspace*{0.2cm} \frac{1}{3!} \left( - \frac{1}{2\pi i} \log \right)^3 \hspace*{0.2cm} &
\hspace*{0.2cm} \frac{1}{2!} \left( - \frac{1}{2\pi i} \log \right)^2 \hspace*{0.2cm} &
\hspace*{0.2cm} -\frac{1}{2 \pi i} \log \hspace*{0.2cm} & \!\cdots \\
\vdots & \vdots & \vdots & \vdots &
\end{array} \right)
\]
We need to know that if we define $\shF^0$ as the span of 
$\shF^0 (\Log_\U) = \, <\!e_0\!>$ and of $f$ (rather than $f \, +$ some non-zero global section of $\Log_\U$), then we get an admissible variation of THS on $\U$.
Modulo a shift of the filtrations, this is the content of \cite{W}, IV,
Theorem 3.5:
\begin{theorem} If we let $f$ be a section of $\shF^0$, then the data define
an admissible variation on $\U$. Because of rigidity, it equals $\pol$.
Therefore, $P$ is the matrix describing $\pol$ in the sense of IV.3.
\end{theorem}
For the description of $\pol$ in $l$-adics, see \cite{B2}, 3.3, or \cite{W},
IV.4.

\section{Cyclotomic elements, and special values}
The talk given at the workshop concerned itself with two further properties
of our objects: the {\it splitting principle}, and {\it norm compatibility}.
Since the latter plays no strategic r\^ole in the proof of the Main Theorem,
we refer to \cite{W}, pp.\ 224-226 for a detailed account.

\subsection{The splitting principle}
Splitting over roots of unity is a property of the logarithmic sheaf $\Log$
rather than of $\pol$. Let $\omega \in \G_m (\C)_{\rtors}$, and consider the
mixed Tate-Hodge structure $\omega^* \Log$.
\begin{prop} $\omega^* \Log$ splits canonically:
\[
\omega^* \Log = \prod_{j \ge 0} \Q(j) \; .
\]
\end{prop}
\begin{proof} One can either employ the universal property \ref{III.1.4} of
$\Log$ to deduce a canonical isomorphism
\[
\Log \silo [n]^* \Log \; ,
\]
where $[n]: \G_m \pfeil \G_m, \; t \longmapsto t^n$. Since $1^* \Log$ is
split, so is the fibre of $\Log$ at any preimage of $1$ under $[n]$.\\

Or use Lemma \ref{II.1.1}, and the explicit description of $\Log$.
\end{proof}

\subsection{Cyclotomic elements}
The splitting principle provides us with canonical projections
\[
\pr_{\omega,j}: \omega^* \Log \pfeil \Q(j) \; ,
\]
for any root of unity $\omega$, and any $j \ge 1$. We get an induced map
$( \pr_{\omega,j})_*$ on the level of $\Ext$ groups.
\begin{prop} For any $\omega \ne 1$, we have
\[
( \pr_{\omega,j})_* (\omega^* \pol) = 
 (-1)^j \Li_j (\omega) \mod (2 \pi i)^j \Q 
\]
in $\Ext^1_{\MHS} (\Q(0), \Q(j)) = \C / (2 \pi i)^j \Q$.
\end{prop}
\begin{proof} Look at the matrix $P(\omega)$!
\end{proof}

\section{$K$-theory}
The general motivic philosophy says (among other things):
\begin{conj}[Beilinson et. al.]
For all varieties over $\Q$, there is a universal cohomology theory $X\mapsto h^*(X)$ with
values in an abelian category $\Mh\Mh$ (mixed motives) and a universal cohomology theory
with values in $\Mh\Mh$. There is also a universal absolute cohomology
theory (motivic cohomology) such that the Leray spectral sequence gives
short exact sequences
\[
 0\to \Ext^1_{\Mh\Mh}(\Q(0),h^{n-1}(X)(j))\to
\hm^n(X,j)\to\Hom_{\Mh\Mh}(\Q(0),h^n(X)(j))\to 0.
\]
Moreover, for smooth varieties $X$, there should be natural isomorphisms
\[\hm^i(X,j)\isom\Gr^j_\gamma K_{2j-i}(X)_\Q\; .\]
\end{conj}
This leads us to {\em define}
\begin{defn}
For smooth $X$ over $\Z$, we put $\hm^i(X,j)\isom\Gr^j_\gamma K_{2j-i}(X)_\Q$.
We call this motivic cohomology of $X$.
\end{defn}
We need to extend this definition. We want
\begin{itemize}
\item relative motivic cohomology groups,
\item motivic cohomology of certain singular varieties,
\item localization sequences in this context.
\end{itemize}
We use the approach of Gillet and Soul\'e \cite{GS} also used in \cite{dJ}
by de Jeu. Details can be found in \cite{HW} Appendix B. The following
is a quick and very imprecise overview.

\subsection{Generalized cohomology}
We recall that the geometric realization functor induces an equivalence of categories between simplicial sets and $CW$-complexes up to homotopy. On the
other hand, the `associated complex' functor gives an equivalence of
categories between simplicial abelian groups and cohomological
complexes concentrated in negative degrees, both up to homotopy.
We sheafify these notions for the Zariski-topology. By this we mean the big
or small site of smooth schemes over a fixed smooth $\Z$-scheme $S_0$, equipped
with the Zariski-topology.
\begin{defn}[\cite{HW} B.1]
A {\em space} $Y$ is a simplicial sheaf of sets for the Zariski-topology which
is pointed by a map $\ast\to Y$. Here $\ast$ is the constant simplicial
object associated to the constant sheaf $S\mapsto \{\ast\}$. A morphism
$y:Y\to Y'$ is called {\em weak equivalence} if the induced morphisms
on the sheafified homotopy groups are isomorphisms for all choices of
base point.
\end{defn}

Spaces form a closed model category in the sense of Quillen. This means that
they behave `like topological spaces'. In particular, we can form suspensions $S$,
cones, loop spaces and form a homotopy category by formally inverting
weak euqivalence.

If $X$ is a scheme, let $\widetilde{X}$ be the constant simplicial object
associated to $S\mapsto X(S)\cup\{\ast\}$ pointed by the disjoint copy of
$\ast$.

\begin{defn} [\cite{HW} B.1.1]
A space $Y$ is {\em constructed from schemes} if all $Y_n$ are
of the form $\ast\cup$ scheme. We define {\em generalized cohomology} of $Y$
with coefficients in a space $A$ by 
\[ \hgen^{-m}(Y,A)=[S^mY,A]\hspace{3em}\text{for $m\geq 0$}\]
where $[\cdot,\cdot ]$ denotes morphisms in the homotopy category.
\end{defn}
In particular, we use $A=\k$, the sheafification of
$S\mapsto \Z\times\Z_\infty BGl(S)$. We speak of $K$-cohomology.
\begin{prop}[Brown-Gersten]
If $Y=\widetilde{X}$ for a scheme $X$ in the site, then 
$\hgen^{-m}(\tilde{X},\k)=K_m(X)$.
\end{prop}
Cohomology of abelian sheaves, e.g., absolute Hodge cohomology, can also be
written as generalized cohomology of a space $K({\cal A})$. 
So one point of generalized cohomology is that it allows to treat $K$-groups
and cohomology of abelian sheaves on an equal footing.
Gillet has
constructed Chern classes $ch_j:\k\to K({\cal A})$ for good
${\cal A}$ like the complexes defining absolute Hodge cohomology. 

\subsection{Motivic cohomology}
Gillet and Soul\'e have constructed maps $\lambda^i:\k\to\k$ for $i\geq 1$
such that $\hgen^i(Y,\k)$ is a turned into a $\lambda$-algebra, at least
if $Y$ is constructed from schemes. (We are lying here! See \cite{HW} B.2.)
Hence we also have a $\gamma$-filtration on $K$-cohomology.
\begin{defn}
If $Y$ is constructed from schemes and $2j+i\leq 0$, let
\[ \hm^i(Y,j)=\Gr^j_\gamma\hgen^{2j+i}(Y,\k)_\Q.\]
\end{defn}
The extension of the definition to spaces allows a lot of extra flexibility.\\

\examples\begin{enumerate}
\item If $T=\bigcup_{i\in I} C_i$ where all $C_i$ and all $\bigcap_{i\in I'} C_i$ for subsets $I'$ of $I$ are smooth over our base $S_0$, then we put
\begin{align*}
T_*&=\cosq_0(\coprod_{i\in I} C_i/T) \hspace{3em}\text{i.e.}\\
T_0& =\coprod_{i\in I} C_i\text{\ \  and\ \ }
T_n=\coprod_{I_n} \bigcap_{i\in I_n} C_i
\end{align*}
where $I_n$ runs through all $n+1$-tuples of elements in $I$. The big
advantage of $T_*$ is that all its components are smooth. 
By adding a disjoint base point we turn this into a space $\widetilde{T}_*$ 
constructed from schemes. 
We {\em define}
$\hm^k(T,j)=\hm^k(\widetilde{T}_*,j)$.
\item Let $T$ as before, $T\subset X$ where $X$ is also smooth. We put
$\hm^i(X\rel T, j)=\hm^i(\cone(\widetilde{T}_*\to \widetilde{X}),j)$ where
the cone is taken in the category of spaces. By definition it sits
in a long exact sequence for relative cohomology.
\end{enumerate}
\begin{theorem}[Soul\'e, de Jeu,\cite{HW} B.2.16] Let $T$ and $X$ be as in
the example. Let $Z\subset X$ be smooth of pure codimension $d$. Suppose
that $Z$ intersects all $\bigcap_{i\in I'} C_i$ transversally. Let $U=X\setminus Z$. Then there is a natural long exact sequence
\[\cdots\to\hm^{i-2d}(Z\rel T\cap Z,j-d)\to \hm^i(X\rel T,j)\to\hm^i(U\rel T\cap U,j)\to\cdots. \]
Moreover, it is compatible with the same sequence in absolute Hodge cohomology
via the Chern class morphism.
\end{theorem}

\section{Motivic polylogarithm}
We now need mixed Hodge modules over $\R$ (\cite{HW} Appendix A). The 
category is denoted $\MHM(X/\R)$. Whereas
admissible variations are the Hodge theoretic version of locally constant
sheaves, Hodge modules correspond to (perverse) constructible sheaves.
\begin{defn}
Let $X$ be a smooth variety over $\R$ and $Y\subset X$ with complement
$j:U\to X$. We define absolute Hodge cohomology by
\begin{gather*}
\habs^i(X/\R,\Q(n))=\Ext^i_{\MHM(X/\R)}(\Q(0),\Q(n))\\
\habs^i(X\rel Y/\R,\Q(n))=\Ext^i_{\MHM(X/\R)}(\Q(0),j_!\Q(n)).
\end{gather*}
\end{defn}
It can be shown that this agrees with Beilinson's ad hoc version (\cite{HW}, Theorem A.2.7). Everything
done in this talk translates immediately into the $l$-adic setting.
\subsection{Geometric origin of $\Log$ and $\pol$}
We now come to a quick tour through \cite{HW}. We consider the following
geometric situation with $\U=\Pe^1\setminus\{0,1,\infty\}$ and $Z=1\times\U\amalg \Delta$, $V=\gm{\U}\setminus Z$, $Z^{(n)}=\gm{\U}^n\setminus V^n$:
\[ {\begin{array}{ccccc}
    V&\xrightarrow{v}{}&\G_{m,\U}&\hookleftarrow Z\\
     &                &\downarrow p\\
     &                &\U
    \end{array}}\hspace{3em}
 {\begin{array}{ccccc}
    V^n&\xrightarrow{v^n}{}&\G^n_{m,\U}&\hookleftarrow Z^{(n)}\\
     &                &\downarrow p^n\\
     &                &\U
    \end{array}}
\]
Recall from lecture III that $\Log_\U^{(1)}=\shG^{(1)}=R^1p_*v_!\Q(1)$ where we
now use  the correct formulation in terms of Hodge modules. Hence
\[ \Log^{(n)}_\U=\Sym^n\Log^{(1)}_\U=\Sym^n R^1p_*v_!\Q(1)
= R^np^n_*v^n_!\Q(n)^\sgn \]
where we have to take the sign-eigenspace with respect to the operation
of the symmetric group because the cup-product is anti-symmetric.
\begin{defn} Let $\shG^{(n)}=R^np^n_*v^n_!\Q(n)^\sgn$.
\end{defn}
With this definition we have (\cite{HW} \S 4)
\begin{align*}
\Ext^1_{\Var(\U/\R)}(\Q(0),\shG^{(n)})&=\Ext^1_{\MHM(\gm{\U}/\R)}(\Q(0),R^np^n_*v^n_!\Q(n)^\sgn)\\
&=\Ext^1_{\MHM(\gm{\U}^n/\R)}(\Q(0),v^n_!\Q(n)^\sgn)\\
&=\habs^{n+1}(\gm{\U}^n\rel Z^{(n)},\Q(n))^\sgn.
\end{align*}
Note that the corresponding motivic cohomology groups
$\hm^{n+1}(\gm{\U}^n\rel Z^{(n)},\Q(n))^\sgn$ are also well-defined!!!
Our main tool in the sequel is the {\em Residue sequence}.
\begin{prop}[\cite{HW} after 4.6 and 7.2]
Let $\bar{Z}=\A^1_\U\setminus V$. There are long exact sequences in
motivic and absolute Hodge cohomology, which are also compatible under
Chern classes:
\begin{multline*}
\cdots\to  H_?^i(\A^n_\U\rel\bar{Z}^{(n)},j)^\sgn\to
  H_?^i(\gm{\U}^n\rel{Z}^{(n)},j)^\sgn\\
\to H_?^{i-1}(\gm{\U}^{n-1}\rel{Z}^{(n-1)},j-1)^\sgn
\to H_?^{i+1}(\A^n_\U\rel\bar{Z}^{(n)},j)^\sgn\to\cdots.
\end{multline*}
\end{prop}
\begin{proof}
In the $n=2$-case we consider the localization sequences in relative cohomology for the two inclusions
$\G_m^2\subset\A^2\setminus (0,0)\subset \A^2$. The effect of the sign-eigenspaces leads to the above form of the sequence.
\end{proof}
\begin{defn} The residue maps are given by the map from the residue sequence
$\res_n:H_?^{n+1}(\gm{\U}^n\rel Z^{(n)},n)\to H_?^{n}(\gm{\U}^{n-1}\rel Z^{(n-1)},n)$.
\end{defn}
\begin{lemma}[\cite{HW} 7.3]
$H^i_?(\A^n_\U\rel\bar{Z}^{(n)},n)\isom H_?^{i-n}(\U,j)$.
\end{lemma}
There is an alternative description of the localization sequence in the
case of absolute Hodge cohomology. The same constructions that led
to the residue sequence, also lead to a sequence of variations on $\U$ by using $\Hh_\U^i(\cdot)=R^ip_*(\cdot)$, namely
\[ 0\to\Q(n)_\U\to \Hh_\U^n(\gm{\U}^n\rel\Z^{(n)},n)^\sgn\xrightarrow{\res_n}{}
  \Hh_\U^{n-1}(\gm{\U}^{n-1}\rel\Z^{(n-1)},n)^\sgn\to 0.\]
Note that $\Hh_\U^n(\gm{\U}^n\rel\Z^{(n)},n)^\sgn$ is just a different
way of writing $\shG^{(n)}$.
\begin{prop}[\cite{HW} 4.9, 4.8]
The square
\[\begin{CD}
\shG^{(n)}@>\res_n>>\shG^{(n-1)}\\
@V\isom VV @VV\isom V\\
\Sym^n\shG^{(1)}@>\text{proj}>> \Sym^{n-1}\shG^{(1)}
\end{CD}\]
commutes. Hence the transition maps of $\Log$ are of geometric origin. Morover
the residue sequence in absolute Hodge cohomology is the long
exact sequence attached to the short exact sequence of sheaves above.
\end{prop}
Up to now everything would have worked for a general base $S$ instead of $\U$.
But now we use the simple form of $\U$. Its cohomology is Tate and by Borel's 
theorem we understand the corresponding motivic cohomology and the regulator
very well.
\begin{lemma}[\cite{HW} 8.3]
With $B=\spec \Z$, the following composition is bijective:
\[ \hm^0(B,0)\xrightarrow{i_1}{}\bigoplus_{i=0,1}\hm^0(B,0)
=\hm^1(\U,1)\to \hm^2(\gm{\U}\rel Z,1).\]
Call the inverse map $\res$, the total residue.
\end{lemma}
\begin{theorem}[\cite{HW} Corollary 8.8]
We have a commutative square:
\[\begin{CD}
\prolim \hm^{n+1}(\gm{\U}\rel Z^{(n)},n)^\sgn@>r_\hodge >>
\habs^1(\U_\R,\Log_\U)\\
@V\res VV@V\text{residue at $1$}V \text{from III}V\\
\hm^0(B,)@>>> \habs^0(B_\R,\Q(0))
\end{CD}\]
The map $\res$ on the left is an isomorphism.
\end{theorem}
We now define $\pol_{\Mh}$, the {\em motivic polylogarithm} simply as $\res^{-1}(1)$.
By construction, $r_\hodge\pol_{\Mh}=\pol$.

\subsection{The motivic splitting principle}
Let $d\geq 2$ and $b$ prime to $d$. Let $C=\spec \Z[T]/\Phi_d(T)[1/d]$. It embeds canonically into $\U$. It can be twisted by raising to the $b$-th power on $C$. Call the resulting embedding $i_b:C\to \U$. The morphism $[d+1]$ on
$\gm{C}$ (raising to the $d+1$-th power) respects $Z_C$. We can analyze
the eigenvalues of this operation and find:
\begin{prop}[\cite{HW} Lemma 9.3]
There is a natural splitting
\[\hm^{n+1}(\gm{C}\rel Z^{(n)}_C,n)=\prod_{1\leq i\leq n}\hm^1(C,i). \]
The splittings are compatible in the projective system and they are also
compatible with the splitting in absolute Hodge cohomology induced
by the splitting principle there.
\end{prop}
We can now prove our main theorem from lecture I:
\begin{defn}
Let 
\begin{align*}
 \epsilon:\{\text{primitive $d$-th roots of unity}\}&\to\hm^1(C,n)\\
    T^b&\mapsto (-1)^{n-1}\frac{1}{n!}\ \text{$n$-component of }i_b^*\pol_{\Mh}.
\end{align*}
\end{defn}
Clearly $r_\hodge\epsilon(\omega)=(-1)^{n-1}\frac{1}{n!}\ \text{$n$-component of }\omega^*\pol$ whose explicit value was computed in lecture V.

\section{Zagier's conjecture}
This talk was devoted to the presentation of the main ideas of the article \cite{BD}.\\

The {\it weak version of Zagier's conjecture}, meanwhile a
theorem of de Jeu's (\cite{dJ}) concerns itself with the $K$-theory of 
number fields. There is a conjecture for any integer $j \ge 1$, and the 
$j$-th can only be formulated if the preceeding ones are true.\\

Fix a number field  $F$. One wants to construct a $\Q$-vector space $\Lh_j$,
a map
\[
\{\;\;\}_j: \U (F) = F^* \setminus \{1\} \longrightarrow \Lh_j \; ,
\]
a homomorphism
\[
d_j: \Lh_j \longrightarrow 
\bigwedge^2 \left( \bigoplus_{l=1}^{j-1} \Lh_l \right) \; ,
\]
and a monomorphism
\[
\varphi_j: \ker(d_j) \hookrightarrow K_{2j-1} (F)_\Q \; .
\]
For $j=1$, one defines $\Lh_1 := F^* \otimes_{\Z} \Q$,
\[
\{x\}_1 := (1-x) \otimes 1 \in \Lh_1 \; ,
\]
$d_1 := 0$, and $\varphi_1$ as the isomorphism between $F^* \otimes_{\Z} \Q$
and $K_1(F)_\Q$.\\

For $j \ge 2$, let $\widetilde{\Lh}_j$ be the free $\Q$-vector space in the
symbols $\widetilde{ \{x\} }_j$, $x \in F^* \setminus \{1\}$. Define
\[
\widetilde{d}_j: \widetilde{\Lh}_j \longrightarrow
\Lh_{j-1} \otimes_{\Q} \Lh_1 \longrightarrow
\bigwedge^2 \left( \bigoplus_{l=1}^{j-1} \Lh_l \right)
\]
by sending the symbol $\widetilde{ \{x\} }_j$ to $\{x\}_{j-1} \wedge x$.\\

The conjecture predicts a map
\[
\widetilde{\varphi}_j: \ker(\widetilde{d}_j) \longrightarrow
K_{2j-1} (F)_\Q \; .
\]
It also predicts the explicit shape of the composition
\[
r_{\Hh} \verk \widetilde{\varphi}_j: \ker(\widetilde{d}_j) \longrightarrow
\left( \bigoplus_{\sigma: F \hookrightarrow \C} 
\C / (2 \pi i)^j \R \right)^+\; .
\]
If the conjecture holds, one sets
\[
\Lh_j := \widetilde{\Lh}_j / \ker(\widetilde{\varphi}_j) \; .
\]
In the talk, it was explained, following \cite{BD}, section 2, that the conjecture follows from the motivic folklore, explained in lecture I, plus
the existence of
``$\pol \in \Mh\Mh^s(\U)$''.
More precisely, if $S \in \ker(\widetilde{d}_j)$, then $\widetilde{\varphi}_j (S)$ is obtained by a linear combination of ``coefficients'' of the value of
$\pol$ at points $x \in \U (F)$.\\

The material covered in lectures II-VII can be seen as a description of
$\widetilde{\varphi}_j$ on a very particular kind of elements of $\ker(\widetilde{d}_j)$, namely those symbols concentrated on roots of unity.


\end{document}